\documentclass[smallextended]{svjour3}       
\smartqed  
\usepackage{amsfonts}
\usepackage{amsmath}
\usepackage{graphicx}
\usepackage[english]{babel}

\newcommand{\Sum}{\displaystyle\sum}

\begin{document}

\title{Sets of values of equivalent almost periodic functions}


\author{J.M. Sepulcre \and
        T. Vidal
} \institute{J.M. Sepulcre  \and T. Vidal   \at
              Department of Mathematics\\
              University of Alicante\\
              03080-Alicante, Spain\\
             \email{JM.Sepulcre@ua.es, tmvg@alu.ua.es}
}

\date{Received: date / Accepted: date}
\maketitle
\begin{abstract}
In this paper we establish a new equivalence relation on the spaces
of almost periodic functions which allows us to prove a result
like Bohr's equivalence theorem extended to the case of all these functions.
\end{abstract}

\section{Introduction}




In the beginnings of the 20th century, the Danish mathematician H. Bohr gave important steps in the understanding
of Dirichlet series, which consist of those exponential sums that take the form
$$\Sum_{n\geq 1}a_ne^{-\lambda_n s},\ a_n\in\mathbb{C},\ s=\sigma+it,$$
 where
$\{\lambda_n\}$ is a strictly increasing sequence of posi\-tive numbers
tending to infinity. 
As a result of his investigations on these functions, he introduced an equivalence relation among them that led to so-called Bohr's equivalence theorem, which shows that equivalent Dirichlet series take the same values in certain vertical lines or strips in the complex plane (e.g. see \cite{Apostol,BohrDirichlet,Rigue,Spira}). 

On the other hand, Bohr also developed during the $1920$'s the theory of almost periodic functions, which opened a way to study a wide class of trigonometric series of the general type and even exponential series (see for example \cite{Besi,Bohr,Bohr2,Corduneanu1,Jessen}). The space of almost periodic functions in a vertical strip $U\subset \mathbb{C}$, which will be denoted in this paper as $AP(U,\mathbb{C})$, coincides with the set of the functions which can be approximated uniformly in every reduced strip 
of $U$ by exponential polynomials $a_1e^{\lambda_1s}+a_2e^{\lambda_2s}+\ldots+a_ne^{\lambda_ns}$ with complex coefficients $a_j$ and real exponents $\lambda_j$ (see for example \cite[Theorem 3.18]{Corduneanu1}).
These aproximating  finite exponential sums can be found by
Bochner-Fej\'{e}r's summation (see, in this regard, \cite[Chapter 1, Section 9]{Besi}). 

On the other hand, the exponential polynomials and the general Dirichlet series are a particular family of exponential sums or, in other words, expressions of the type
$$P_1(p)e^{\lambda_1p}+\ldots+P_j(p)e^{\lambda_jp}+\ldots,$$
where the $\lambda_j$'s are complex numbers and the $P_j(p)$'s are polynomials in the parameter $p$.
In this respect, we established in \cite[Definition 2]{SVBohr} (see also \cite[Definition 3]{SV}) a generalization of Bohr's equivalence relation on the classes $\mathcal{S}_{\Lambda}$ consisting of exponential sums of the form
 \begin{equation}\label{eqqnew}
\sum_{j\geq 1}a_je^{\lambda_jp},\ a_j\in\mathbb{C},\ \lambda_j\in\Lambda,
\end{equation}
where $\Lambda=\{\lambda_1,\lambda_2,\ldots,\lambda_j,\ldots\}$ is an arbitrary countable set of distinct real numbers (not necessarily unbounded), which are called a set of exponents or frequencies. 
Based on this equivalence relation, and under the assumption of existence of an integral basis for the set of exponents (whose condition is defined in Section \ref{sect2}),
we proved in
\cite[Theorem 1]{SVBohr} that two equivalent almost periodic functions, whose associated Dirichlet series have the same set of exponents, take the same values on any open vertical strip included in their strip of almost periodicity $U$, which constitutes a certain extension of Bohr's equivalence theorem for this case of functions.

In this way, the main purpose of this paper is to obtain a result like Bohr's equivalence theorem to all almost periodic functions in $AP(U,\mathbb{C})$, not only for those whose set of frequencies has an integral basis (which was already studied in \cite{SVBohr}) but also in the general case. To do this, we will consider a new equivalence relation $\shortstack{$_{{\fontsize{6}{7}\selectfont *}}$\\$\sim$}$ on the classes $\mathcal{S}_{\Lambda}$ (see Definition \ref{DefEquiv00}), from which we will show that every equivalence class in $AP(U,\mathbb{C})/\shortstack{$_{{\fontsize{6}{7}\selectfont *}}$\\$\sim$}$ 
is connected with a certain auxiliary function that originates the sets of values taken by this equivalence class
along a given vertical line included in the strip of almost periodicity (see Proposition \ref{pult} in this paper). This leads us to formulate and prove Theorem \ref{beqg}, which is the main result of this paper. 

\section{Definitions and preliminary results}\label{sect2}

We first recall the equivalence relation, based on that of \cite[p.173]{Apostol} for general Dirichlet series, which was defined in \cite{SV} in a more general context. 

\begin{definition}\label{DefEquiv}
Let $\Lambda$ be an arbitrary countable subset of distinct real numbers, $V$ the $\mathbb{Q}$-vector space generated by $\Lambda$ ($V\subset \mathbb{R}$), and $\mathcal{F}$
the $\mathbb{C}$-vector space of arbitrary functions $\Lambda\to\mathbb{C}$. 
We define
a relation $\sim$ on $\mathcal{F}$ by $a\sim b$ if there exists a $\mathbb{Q}$-linear map $\psi:V\to\mathbb{R}$ such that
$$b(\lambda)=a(\lambda)e^{i\psi(\lambda)},\ (\lambda\in\Lambda).$$
\end{definition}

From the equivalence relation above, we can consider a new relation on the classes $S_\Lambda$ of exponential sums of type \eqref{eqqnew}. From now on, we will denote as $\sharp A$ the cardinal of a set $A$.

\begin{definition}\label{DefEquiv00}
Given $\Lambda=\{\lambda_1,\lambda_2,\ldots,\lambda_j,\ldots\}$ a set of exponents, consider $A_1(p)$ and $A_2(p)$ two exponential sums in the class $\mathcal{S}_{\Lambda}$, say
$A_1(p)=\sum_{j\geq1}a_je^{\lambda_jp}$ and $A_2(p)=\sum_{j\geq1}b_je^{\lambda_jp}.$
We will say that $A_1$ is equivalent to $A_2$ if for each integer value $n\geq1$, with $n\leq\sharp\Lambda$, it is satisfied
$a_{n}^*\sim b_{n}^*$, where $a_{n}^*,b_{n}^*:\{\lambda_1,\lambda_2,\ldots,\lambda_{n}\}\to\mathbb{C}$ are the functions given by $a_{n}^*(\lambda_j):=a_j$ y $b_{n}^*(\lambda_j):=b_j$, $j=1,2,\ldots,n$ and $\sim$ is in Definition \ref{DefEquiv}.
\end{definition}

It is clear that the relation defined in the foregoing definition is an equivalence relation. 
We will use $\sim$ for the equivalence relation introduced in Definition \ref{DefEquiv} and $\shortstack{$_{{\fontsize{6}{7}\selectfont *}}$\\$\sim$}$ for that of Definition \ref{DefEquiv00}.

Let
$G_{\Lambda}=\{g_1, g_2,\ldots, g_k,\ldots\}$ be a basis of the
$\mathbb{Q}$-vector space generated by a set $\Lambda=\{\lambda_1,\lambda_2,\ldots\}$ of exponents 
(by abuse of notation, we will say that $G_{\Lambda}$ is a basis for $\Lambda$),
which implies that $G_{\Lambda}$ is linearly independent over the rational numbers and each $\lambda_j$ is expressible as a finite linear combination of terms of $G_{\Lambda}$, say
\begin{equation}\label{errej}
\lambda_j=\sum_{k=1}^{i_j}r_{j,k}g_k,\ \mbox{for some }r_{j,k}\in\mathbb{Q},\ i_j\in\mathbb{N}.
\end{equation}
We will say that $G_{\Lambda}$ is an \textit{integral basis} for $\Lambda$ when $r_{j,k}\in\mathbb{Z}$ for any $j,k$.
Moreover, we will say that $G_{\Lambda}$ is the \textit{natural basis} for $\Lambda$, and we will denote it as $G_{\Lambda}^*$, when it is constituted by elements in $\Lambda$. That is, firstly if $\lambda_1\neq 0$ then $g_1:=\lambda_1\in G_{\Lambda}^*$. Secondly, if $\{\lambda_1,\lambda_2\}$ are $\mathbb{Q}$-rationally independent, then $g_2:=\lambda_2\in G_{\Lambda}^*$. Otherwise, if $\{\lambda_1,\lambda_3\}$ are $\mathbb{Q}$-rationally independent, then $g_2:=\lambda_3\in G_{\Lambda}^*$, and so on. In this way, 
if $\lambda_j\in G_{\Lambda}^*$ then $r_{j,m_j}=1$ and $r_{j,k}=0$ for $k\neq m_j$, where $m_j$ is such that $g_{m_j}=\lambda_j$. In fact, each element in $G_{\Lambda}^*$ is of the form $g_{m_j}$ for $j$ such that $\lambda_j$ is $\mathbb{Q}$-linear independent of the previous elements in the basis. 
Furthermore, if $\lambda_j\notin G_{\Lambda}^*$ then $\lambda_j=\sum_{k=1}^{i_j}r_{j,k}g_k$, where $\{g_{1},g_{2},\ldots,g_{i_j}\}\subset \{\lambda_1,\lambda_2,\ldots,\lambda_{j-1}\}$.

In terms of a prefixed basis for the set of exponents $\Lambda$, we next quote a first characterization of the equivalence of two exponential sums in $\mathcal{S}_{\Lambda}$ \cite[Proposition 1 (mod.)]{SV}.

\begin{proposition}\label{prop1prima}
Given $\Lambda=\{\lambda_1,\lambda_2,\ldots,\lambda_j,\ldots\}$ a set of exponents, consider $A_1(p)$ and $A_2(p)$ two exponential sums in the class $\mathcal{S}_{\Lambda}$, say
$A_1(p)=\sum_{j\geq1}a_je^{\lambda_jp}$ and $A_2(p)=\sum_{j\geq1}b_je^{\lambda_jp}.$
Fixed a basis $G_{\Lambda}$ for $\Lambda$, for each $j\geq1$ let $\mathbf{r}_j\in \mathbb{R}^{\sharp G_{\Lambda}}$ be the vector of rational components verifying (\ref{errej}).  
Then $A_1\shortstack{$_{{\fontsize{6}{7}\selectfont *}}$\\$\sim$} A_2$ 
if and only if for each integer value $n\geq1$, with $n\leq\sharp\Lambda$, there exists a vector $\mathbf{x}_n=(x_{n,1},x_{n,2},\ldots,x_{n,k},\ldots)\in \mathbb{R}^{\sharp G_{\Lambda}}$
such that $b_j=a_j e^{<\mathbf{r}_j,\mathbf{x}_n>i}$ for $j=1,2,\ldots,n$.\vspace{0.1cm} 

\noindent Furthermore, if $G_{\Lambda}$ is an integral basis for $\Lambda$ then $A_1\shortstack{$_{{\fontsize{6}{7}\selectfont *}}$\\$\sim$} A_2$ 
if and only if there exists $\mathbf{x}_0=(x_{0,1},x_{0,2},\ldots,x_{0,k},\ldots)\in \mathbb{R}^{\sharp G_{\Lambda}}$
such that $b_j=a_j e^{<\mathbf{r}_j,\mathbf{x}_0>i}$ for every $j\geq 1$. 
\end{proposition}

Now, in terms of the natural basis for a set of exponents $\Lambda$, we next provide a second characterization of the equivalence of two exponential sums in $\mathcal{S}_{\Lambda}$.

\begin{proposition}\label{propnaturalbasis}
Given $\Lambda=\{\lambda_1,\lambda_2,\ldots,\lambda_j,\ldots\}$ a set of exponents, consider $A_1(p)$ and $A_2(p)$ two exponential sums in the class $\mathcal{S}_{\Lambda}$, say
$A_1(p)=\sum_{j\geq1}a_je^{\lambda_jp}$ and $A_2(p)=\sum_{j\geq1}b_je^{\lambda_jp}.$
Fixed the natural basis $G_{\Lambda}^*=\{g_1,g_2,\ldots,g_k,\ldots\}$ for $\Lambda$, for each $j\geq1$ let $\mathbf{r}_j\in \mathbb{R}^{\sharp G_{\Lambda}^*}$ be the vector of rational components verifying (\ref{errej}).  
Then $A_1\shortstack{$_{{\fontsize{6}{7}\selectfont *}}$\\$\sim$} A_2$ 
if and only if there exists $\mathbf{x}_0=(x_{0,1},x_{0,2},\ldots,x_{0,k},\ldots)\in[0,2\pi)^{\sharp G_{\Lambda}^*}$
such that for each $j=1,2,\ldots$ it is satisfied $b_j=a_j e^{<\mathbf{r}_j,\mathbf{x}_0+\mathbf{p}_j>i}$ for some $\mathbf{p}_j=(2\pi n_{j,1},2\pi n_{j,2},\ldots)\in \mathbb{R}^{\sharp G_{\Lambda}^*}$, with $n_{j,k}\in\mathbb{Z}$.
\end{proposition}
\begin{proof}
Suppose that $A_1\shortstack{$_{{\fontsize{6}{7}\selectfont *}}$\\$\sim$} A_2$. Consider $I=\{1,2,\ldots,k,\ldots: \lambda_k\in G_{\Lambda}^*\}$ and $I_n=\{1,2,\ldots,k,\ldots,n: \lambda_k\in G_{\Lambda}^*\}$.
Let $j\in I$, then $r_{j,m_j}=1$ and $r_{j,k}=0$ for $k\neq m_j$, where $m_j$ is such that $g_{m_j}=\lambda_j$.
Thus, by Proposition \ref{prop1prima}, let $\mathbf{x}_j=(x_{j,1},x_{j,2},\ldots)\in\mathbb{R}^{\sharp G_{\Lambda}^*}$ be a vector such that
\begin{equation}\label{hfwsc}
b_j=a_j e^{i<\mathbf{r}_j,\mathbf{x}_j>}=a_j e^{i\sum_{k=1}^{i_j}r_{j,k}x_{j,k}}=a_je^{ir_{j,m_j}x_{j,m_j}}=a_je^{ix_{j,m_j}}.
\end{equation}
Define $\mathbf{x}_0=(x_{0,1},x_{0,2},\ldots)\in\mathbb{R}^{\sharp G_{\Lambda}^*}=\mathbb{R}^{\sharp I}$ as $x_{0,m_j}:=x_{j,m_j}$ for $j\in I$. Thus, by taking $\mathbf{p}_j=(0,0,\ldots)$, the result trivially holds for those $j$'s such that $\lambda_j\in G_{\Lambda}^*$, i.e. for $j\in I$. Now, let $j$ be such that $\lambda_j\notin G_{\Lambda}^*$, i.e. $j\notin I$. By Proposition \ref{prop1prima}, let $\mathbf{x}_j=(x_{j,1},x_{j,2},\ldots)\in\mathbb{R}^{\sharp G_{\Lambda}^*}$ be a vector such that
$$b_p=a_p e^{i<\mathbf{r}_p,\mathbf{x}_j>}=a_p e^{i\sum_{k=1}^{i_j}r_{p,k}x_{j,k}},\ p=1,2,\ldots,j.$$
Note that if $p=1,2,\ldots,j$ is such that $\lambda_p\in G_{\Lambda}^*$, then
$$b_p=a_pe^{ir_{p,m_p}x_{j,m_p}}=a_pe^{ix_{j,m_p}},$$
which necessarily implies, by (\ref{hfwsc}), that $x_{j,m_p}=x_{p,m_p}+2\pi n_{j,p}$ for some $n_{j,p}\in\mathbb{Z}$. 
Hence
$$b_j=a_j e^{i<\mathbf{r}_j,\mathbf{x}_j>}=a_j e^{i\sum_{k=1}^{i_j}r_{j,k}x_{j,k}}=a_j e^{i\sum_{p\in I_{j-1}}r_{j,m_p}x_{j,m_p}}=$$
$$a_j e^{i\sum_{p\in I_{j-1}}r_{j,m_p}(x_{p,m_p}+2\pi n_{j,p})}=a_j e^{i<\mathbf{r}_j,\mathbf{x}_0+\mathbf{p}_j>},$$
where $\mathbf{p}_j=(2\pi n_{j,1},2\pi n_{j,2},\ldots,0,0,\ldots)$. Moreover, by changing conveniently the vectors $\mathbf{p}_j$, we can take $\mathbf{x}_0\in [0,2\pi)^{\sharp G_{\Lambda}^*}$ without loss of generality.

Conversely, suppose the existence of $\mathbf{x}_0=(x_{0,1},x_{0,2},\ldots,x_{0,k},\ldots)\in\mathbb{R}^{\sharp G_{\Lambda}^*}$ satisfying $b_j=a_j e^{<\mathbf{r}_j,\mathbf{x}_0+\mathbf{p}_j>i}$ for some $\mathbf{p}_j=(2\pi n_{j,1},2\pi n_{j,2},\ldots)\in \mathbb{R}^{\sharp G_{\Lambda}^*}$, with $n_{j,k}\in\mathbb{Z}$. Let $r_{j,k}=\frac{p_{j,k}}{q_{j,k}}$ with $p_{j,k}$ and $q_{j,k}$ coprime integer numbers, and define $q_{n,k}:=\operatorname{lcm}(q_{1,k},q_{2,k},\ldots,q_{n,k})$ for each $k=1,2,\ldots$.
Thus, for any integer number $n\geq 1$, take $\mathbf{x}_n=\mathbf{x}_0+\mathbf{m}_n$, where  $m_{n,k}=2\pi p_{1,k}p_{2,k}\cdots p_{n,k}q_{n,k}$, $k=1,2,\ldots$.
Therefore,
it is satisfied $b_j=a_j e^{<\mathbf{r}_j,\mathbf{x}_n>i}$ for each $j=1,2,\ldots,n$, which implies that $A_1\shortstack{$_{{\fontsize{6}{7}\selectfont *}}$\\$\sim$} A_2$.
\qed
\end{proof}


%

We next study the case where the chosen basis is not the natural one.
Fixed a set $\Lambda=\{\lambda_1,\lambda_2,\ldots\}$ of exponents, let $G_{\Lambda}^*$ be the natural basis for $\Lambda$ and $G_{\Lambda}$ be an arbitrary basis for $\Lambda$. For each $j\geq 1$ let $\mathbf{r}_j$ and $\mathbf{s}_j$ be the vectors of rational components so that  $\lambda_j=<\mathbf{r}_j,\mathbf{g}>$ and $\lambda_j=<\mathbf{s}_j,\mathbf{h}>$, with $\mathbf{g}$ and $\mathbf{h}$ the vectors associated with the basis $G_{\Lambda}^*$ and $G_{\Lambda}$, respectively. Finally, for each $k\geq 1$, let $\mathbf{t}_k$ be the vector so that $h_k=<\mathbf{t}_k,\mathbf{g}>$, i.e.
\begin{equation}\label{matrix}
T=\begin{pmatrix}
t_{1,1} & t_{1,2} & \cdots & t_{1,j} & \cdots \\
t_{2,1} & t_{2,2} & \cdots & t_{2,j} & \cdots\\
\vdots & \dots & \ddots & \vdots & \cdots\\
t_{k,1} & t_{k,2} & \cdots & t_{k,j} & \cdots\\
\vdots & \dots & \ddots & \vdots & \cdots\\
\end{pmatrix}\end{equation}
is the change of basis matrix.
Thus, for any $\mathbf{x}_0\in\mathbb{R}^{\sharp G_{\Lambda}}$, we have
$$<\mathbf{r}_j,\mathbf{x}_0>=<\mathbf{s}_j,\mathbf{x}_1>,$$
where $\mathbf{x}_1$ is defined as $x_{1,k}=<\mathbf{t}_k,\mathbf{x}_0>$ for each $k\geq 1$. Indeed,
$$<\mathbf{s}_j,\mathbf{x}_1>=\sum_{k}s_{j,k}x_{1,k}=\sum_k s_{j,k}<\mathbf{t}_k,\mathbf{x}_0>=$$$$\sum_k s_{j,k}\sum_m t_{k,m}x_{0,m}=\sum_mx_{0,m}\sum_k s_{j,k} t_{k,m}=\sum_m r_{j,m}x_{0,m}=<\mathbf{r}_j,\mathbf{x}_0>$$
because $r_{j,m}=\sum_k s_{j,k}t_{k,m}$.
Consequently, for any $j$ and
\begin{equation}\label{vector}
\mathbf{p}_j=(2\pi n_{j,1},2\pi n_{j,2},\ldots)\in \mathbb{R}^{\sharp G_{\Lambda}},\ \mbox{with }n_{j,k}\in\mathbb{Z},
\end{equation}
we have
$$<\mathbf{r}_j,\mathbf{x}_0+\mathbf{p}_j>=<\mathbf{r}_j,\mathbf{x}_0>+<\mathbf{r}_j,\mathbf{p}_j>=<\mathbf{s}_j,\mathbf{x}_1>+<\mathbf{s}_j,\mathbf{q}_j>=<\mathbf{s}_j,\mathbf{x}_1+\mathbf{q}_j>,$$
where $\mathbf{q}_j$ is defined as $q_{1,k}=<\mathbf{t}_k,\mathbf{p}_j>$ for each $k\geq 1$, i.e. $\mathbf{q}_j$ is obtained from $T\cdot \mathbf{p}_j^t$.
In this way, we have proved the following result.

\begin{corollary}\label{cor1ner}
Given $\Lambda=\{\lambda_1,\lambda_2,\ldots,\lambda_j,\ldots\}$ a set of exponents, consider $A_1(p)$ and $A_2(p)$ two exponential sums in the class $\mathcal{S}_{\Lambda}$, say
$A_1(p)=\sum_{j\geq1}a_je^{\lambda_jp}$ and $A_2(p)=\sum_{j\geq1}b_je^{\lambda_jp}.$
Fixed a basis $G_{\Lambda}=\{g_1,g_2,\ldots,g_k,\ldots\}$ for $\Lambda$, for each $j\geq1$ let $\mathbf{r}_j\in \mathbb{R}^{\sharp G_{\Lambda}}$ be the vector of rational components verifying (\ref{errej}).  
Then $A_1\shortstack{$_{{\fontsize{6}{7}\selectfont *}}$\\$\sim$} A_2$ 
if and only if there exists $\mathbf{x}_0=(x_{0,1},x_{0,2},\ldots,x_{0,k},\ldots)\in[0,2\pi)^{\sharp G_{\Lambda}}$
such that for each $j=1,2,\ldots$ it is satisfied $b_j=a_j e^{<\mathbf{r}_j,\mathbf{x}_0+\mathbf{q}_j>i}$ for some $\mathbf{q}_j\in \mathbb{R}^{\sharp G_{\Lambda}}$ which is of the form
$\mathbf{q}_j=\mathbf{p}_j\cdot T^t$, where $T^t$ is the transpose of the change of basis matrix (\ref{matrix}) and $\mathbf{p}_j$ is of the form (\ref{vector}).
\end{corollary}


We next construct a generating expression for all exponential polynomials in a class $\mathcal{G}\subset \mathcal{S}_{\Lambda}/\shortstack{$_{{\fontsize{6}{7}\selectfont *}}$\\$\sim$}$. Let $2\pi \mathbb{Z}^{m}=\{(c_1,c_2,\ldots,c_m)\in\mathbb{R}^{m}:c_k=2\pi n_k, \mbox{with }n_k\in\mathbb{Z},\ k=1,2,\ldots,m\}$.
From Proposition \ref{propnaturalbasis}, it is clear that the set of all exponential sums $A(p)$ in an equivalence class $\mathcal{G}$ in $\mathcal{S}_{\Lambda}/\shortstack{$_{{\fontsize{6}{7}\selectfont *}}$\\$\sim$}$ can be determined by a function $E_{\mathcal{G}}:[0,2\pi)^{\sharp G_{\Lambda}^*}\times \prod_{j\geq 1}2\pi\mathbb{Z}^{\sharp G_{\Lambda}^*}\rightarrow \mathcal{S}_{\Lambda}$ of the form
\begin{equation}\label{2.4.000}
E_{\mathcal{G}}(\mathbf{x},\mathbf{p}_1,\mathbf{p}_2,\ldots):=\sum_{j\geq1}a_je^{<\mathbf{r}_j,\mathbf{x}+\mathbf{p}_j>i}e^{\lambda_jp}\text{, }
\mathbf{x}\in%
[0,2\pi)^{\sharp G_{\Lambda}^*},\ \mathbf{p}_j\in 2\pi\mathbb{Z}^{\sharp G_{\Lambda}^*},
\end{equation}%
where $a_1,a_2,\ldots,a_j,\ldots$ are the coefficients of an exponential sum in $\mathcal{G}$ and the $\mathbf{r}_j$'s are the vectors of rational components associated with the natural basis $G_{\Lambda}^*$ for $\Lambda$. 
We recommend the reader compare the definition of the function $E_{\mathcal{G}}$ with that of \cite[Expression (3)]{SVBohr}.

In particular, in this paper we are going to use Definition \ref{DefEquiv00} for the case of exponential sums in $\mathcal{S}_{\Lambda}$ of a complex variable $s=\sigma+it$. 
Precisely,
when the formal series in $\mathcal{S}_{\Lambda}$ are handled as exponential sums of a complex variable on which we fix a summation procedure, from equivalence class generating expression (\ref{2.4.000}) we can consider an auxiliary function as follows (compare with \cite[Definition 3]{SVBohr}).

\begin{definition}\label{auxuliaryfunc0}
Given $\Lambda=\{\lambda_1,\lambda_2,\ldots,\lambda_j,\ldots\}$ a set of exponents, let $\mathcal{G}$ be an equivalence class in $\mathcal{S}_{\Lambda}/\shortstack{$_{{\fontsize{6}{7}\selectfont *}}$\\$\sim$}$ and $a_1,a_2,\ldots,a_j,\ldots$ be the coefficients of an exponential sum in $\mathcal{G}$.
For each $j\geq 1$ let $\mathbf{r}_j$ be the vector of rational components satis\-fying the equality $\lambda_j=<\mathbf{r}_j,\mathbf{g}>=\sum_{k=1}^{q_j}r_{j,k}g_k$, where
$\mathbf{g}:=(g_1,\ldots,g_k,\ldots)$ is the vector of the elements of the natural basis $G_{\Lambda}^*$ for $\Lambda$. Suppose that some elements in $\mathcal{G}$, handled as exponential sums of a complex variable $s=\sigma+it$, are summable on at least  a certain set $P$ included in the real axis by some prefixed summation method.
Then we define the auxiliary function  $F_{\mathcal{G}}: P
\times
[0,2\pi)
^{\sharp G_{\Lambda}^*}\times \prod_{j\geq 1}2\pi\mathbb{Z}^{\sharp G_{\Lambda}^*}\rightarrow
\mathbb{C}$
 associated with $\mathcal{G}$, relative to the basis $G_{\Lambda}^*$, as
\begin{equation}\label{2.4.000p}
F_{\mathcal{G}}(\sigma,\mathbf{x},\mathbf{p}_1,\mathbf{p}_2,\ldots):=\sum_{j\geq1}a_je^{<\mathbf{r}_j,\mathbf{x}+\mathbf{p}_j>i}e^{\lambda_j\sigma}\text{, }
\end{equation}%
where $\sigma\in P,\ \mathbf{x}\in%
[0,2\pi)^{\sharp G_{\Lambda}^*}$, $\mathbf{p}_k\in 2\pi\mathbb{Z}^{\sharp G_{\Lambda}^*},$
and the series in (\ref{2.4.000p}) is summed by the prefixed summation method, applied at $t=0$ to the exponential sum obtained from the generating expression (\ref{2.4.000}) with $p=\sigma+it$.
\end{definition}

In particular, see Definition \ref{auxuliaryfunc} which concerns the case of almost periodic functions with the Bochner-Fej\'{e}r summation method and the set $P$ above is formed by the real projection of the strip of almost periodicity of the corresponding exponential sums.
In this context, we will see in this paper the strong link
between the sets of values in the complex plane taken by a function in $AP(U,\mathbb{C})$, its Dirichlet series and its associated auxiliary function.



\begin{definition}\label{DF}
Let $\Lambda=\{\lambda_1,\lambda_2,\ldots,\lambda_j,\ldots\}$ be an arbitrary countable set of distinct real numbers. We will say that a function $f:U\subset\mathbb{C}\to\mathbb{C}$ 
is in the class $\mathcal{D}_{\Lambda}$ if it is an almost periodic function in $AP(U,\mathbb{C})$ 
whose associated Dirichlet series 
is of the form
 \begin{equation}\label{eqqo}
\sum_{j\geq 1}a_je^{\lambda_js},\ a_j\in\mathbb{C},\ \lambda_j\in\Lambda,
\end{equation}
where $U$ is a strip of the type $\{s\in\mathbb{C}: \alpha<\operatorname{Re}s<\beta\}$, with $-\infty\leq\alpha<\beta\leq\infty$.
\end{definition}

Each almost periodic function in $AP(U,\mathbb{C})$ 
is determined by its Dirichlet series, 
which is of type (\ref{eqqo}). 
In fact it is convenient to remark that, even in the case that the sequence of the partial sums of its Dirichlet series 
does not converge uniformly, there exists a sequence of finite exponential sums, the Bochner-Fej\'{e}r polynomials, of the type $P_k(s)=\sum_{j\geq 1}p_{j,k}a_je^{\lambda_js}$ 
where for each $k$ only a finite number of the factors $p_{j,k}$ differ from zero, which converges uniformly to $f$ in every reduced strip in $U$ 
and converges formally to the Dirichlet series  
\cite[Polynomial approximation theorem, pgs. 50,148]{Besi}.

Moreover, the equivalence relation of Definition \ref{DefEquiv00} can be immediately adapted to the case of the functions (or classes of functions) which are identifiable by their also called Dirichlet series, in particular to the classes $\mathcal{D}_{\Lambda}$. 
More specifically, see \cite[Section 4, Definition 5 (mod.)]{SV} referred to the Besicovitch space which contains the classes of functions which are associated with Fourier or Dirichlet series and for which the extension of our equivalence relation makes sense. 

\section{The auxiliary functions associated with the classes $\mathcal{D}_{\Lambda}$}\label{sv}

Based on Definition \ref{auxuliaryfunc0}, applied to our particular case of almost periodic functions with the Bochner-Fej\'{e}r summation method, note that to every almost periodic function $f\in \mathcal{D}_{\Lambda}$, with $\Lambda$ an arbitrary set of exponents, we can associate an auxiliary function $F_f$ of countably many real variables as follows.

\begin{definition}\label{auxuliaryfunc}
Given $\Lambda=\{\lambda_1,\lambda_2,\ldots,\lambda_j,\ldots\}$ a set of exponents, let $f(s)\in\mathcal{D}_{\Lambda}$ be an almost periodic function in $\{s\in\mathbb{C}:\alpha<\operatorname{Re}s<\beta\}$, $-\infty\leq\alpha<\beta\leq\infty$, whose Dirichlet series is given by $\sum_{j\geq 1}a_je^{\lambda_js}$.
For each $j\geq 1$ let $\mathbf{r}_j$ be the vector of rational components satisfying the equality $\lambda_j=<\mathbf{r}_j,\mathbf{g}>=\sum_{k=1}^{q_j}r_{j,k}g_k$, where
$\mathbf{g}:=(g_1,\ldots,g_k,\ldots)$ is the vector of the elements of the natural basis $G_{\Lambda}^*$ for $\Lambda$.
Then we define the auxiliary function  $F_f: (\alpha,\beta)
\times
[0,2\pi)
^{\sharp G_{\Lambda}^*}\times \prod_{j\geq 1}2\pi\mathbb{Z}^{\sharp G_{\Lambda}^*}\rightarrow
\mathbb{C}$
 associated with $f$, relative to the basis $G_{\Lambda}^*$, as
\begin{equation}\label{2.4}
F_{f}(\sigma,\mathbf{x},\mathbf{p}_1,\mathbf{p}_2,\ldots):=\sum_{j\geq1}a_j e^{\lambda_j\sigma
}e^{<\mathbf{r}_j,\mathbf{x}+\mathbf{p}_j>i}\text{, }
\end{equation}%
where $\sigma \in
(\alpha,\beta)
\text{, }\mathbf{x}\in%
[0,2\pi)^{\sharp G_{\Lambda}^*},\ \mathbf{p}_j\in2\pi\mathbb{Z}^{\sharp G_{\Lambda}^*}$ and series (\ref{2.4}) is summed by Bochner-Fej\'{e}r procedure, applied at $t=0$ to the sum
$\sum_{j\geq1}a_j e^{<\mathbf{r}_j,\mathbf{x}+\mathbf{p}_j>i}e^{\lambda_js}$.
\end{definition}

If $f\in AP(U,\mathbb{C})$, it was proved in \cite[Lemma 3]{SV} (see also the Arxiv version) 
that each function of its equivalence class is also included in $AP(U,\mathbb{C})$. Then we first note that, if $\sum_{j\geq1}a_j e^{\lambda_js}$ is the Dirichlet series of $f\in AP(U,\mathbb{C})$, for every choice of $\mathbf{x}\in[0,2\pi)^{\sharp G_{\Lambda}}$ and $\mathbf{p}_j\in 2\pi\mathbb{Z}^{\sharp G_{\Lambda}}$, $j=1,2,\ldots$, the sum $\sum_{j\geq1}a_j e^{<\mathbf{r}_j,\mathbf{x}+\mathbf{p}_j>i}e^{\lambda_js}$ represents the Dirichlet series of an almost periodic function. 

We second note that if the Dirichlet series of $f(s)\in AP(U,\mathbb{C})$ converges uniformly on $U=\{s\in\mathbb{C}:\alpha<\operatorname{Re}s<\beta\}$, then $f(s)$ coincides with its Dirichlet series and (\ref{2.4}) can be viewed as summation by partial sums or ordinary summation.



In addition, we third note that the Dirichlet series $\sum_{j\geq 1}a_je^{\lambda_js}$, associated with a certain function $f\in\mathcal{D}_{\Lambda}$  arises from its auxiliary function $F_f$ by a special choice of its variables, that is $F_f(\sigma,t\mathbf{g},\mathbf{0},\mathbf{0},\ldots)=\sum_{j\geq 1}a_je^{\lambda_j(\sigma+it)}$. In fact, as we will see in this section, there is a strong link between the sets of values in the complex plane taken by both functions. 

In this respect, under the assumption that the natural basis for the set of the exponents is also an integral basis, it is clear that the vectors $\mathbf{p}_j$ do not play any role and hence the auxiliary function $F_f$,
 associated with $f$, can be taken as $F_{f}(\sigma,\mathbf{x}):=\sum_{j\geq1}a_j e^{\lambda_j\sigma
}e^{<\mathbf{r}_j,\mathbf{x}>i}$, $\sigma \in
(\alpha,\beta)
\text{, }\mathbf{x}\in%
[0,2\pi)^{\sharp G_{\Lambda}^*}$.

In general, if we consider an arbitrary basis for the set of exponents, Definition \ref{auxuliaryfunc} can be adapted by taking into account Corollary \ref{cor1ner}. For this purpose, given a basis $G_{\Lambda}$ for $\Lambda$, let $T$ be the change of basis matrix (\ref{matrix}), with respect to the natural basis, and
let $$S_T=\{\mathbf{q}\in \mathbb{R}^{\sharp G_{\Lambda}}: \mathbf{q}=\mathbf{p}\cdot T^t,\ \mbox{with }\mathbf{p}\mbox{ of the form }(\ref{vector})\}.$$

\begin{definition}\label{auxuliaryfunc2}
Given $\Lambda=\{\lambda_1,\lambda_2,\ldots,\lambda_j,\ldots\}$ a set of exponents, let $f(s)\in\mathcal{D}_{\Lambda}$ be an almost periodic function in $\{s\in\mathbb{C}:\alpha<\operatorname{Re}s<\beta\}$, $-\infty\leq\alpha<\beta\leq\infty$, whose Dirichlet series is given by $\sum_{j\geq 1}a_je^{\lambda_js}$.
For each $j\geq 1$ let $\mathbf{s}_j$ be the vector of rational components satisfying the equality $\lambda_j=<\mathbf{s}_j,\mathbf{g}>=\sum_{k=1}^{q_j}s_{j,k}g_k$, where
$\mathbf{g}:=(g_1,\ldots,g_k,\ldots)$ is the vector of the elements of an arbitrary basis $G_{\Lambda}$ for $\Lambda$.
Then we define the auxiliary function  $F_f^{G_{\Lambda}}: (\alpha,\beta)
\times
[0,2\pi)
^{\sharp G_{\Lambda}}\times \prod_{j\geq 1}S_T\rightarrow
\mathbb{C}$
 associated with $f$, relative to the basis $G_{\Lambda}$, as
\begin{equation}\label{2.4.11}
F_f^{G_{\Lambda}}(\sigma,\mathbf{x},\mathbf{q}_1,\mathbf{q}_2,\ldots):=\sum_{j\geq1}a_j e^{\lambda_j\sigma
}e^{<\mathbf{s}_j,\mathbf{x}+\mathbf{q}_j>i}\text{, }
\end{equation}%
where $\sigma \in
(\alpha,\beta)
\text{, }\mathbf{x}\in%
[0,2\pi)^{\sharp G_{\Lambda}}$, $\mathbf{q}_j\in S_T$, and
 series (\ref{2.4.11}) is summed by Bochner-Fej\'{e}r procedure, applied at $t=0$ to the sum
$\sum_{j\geq1}a_j e^{<\mathbf{r}_j,\mathbf{x}>i}e^{\lambda_js}$.
\end{definition}

If we take the natural basis, it is obvious that $F_f=F_f^{G_{\Lambda}^*}$.


We next show a characterization of the property of equivalence of functions in the classes $\mathcal{D}_{\Lambda}$ in terms of the auxiliary function relative to the natural basis.


\begin{proposition}\label{lequiv20}
Given $\Lambda=\{\lambda_1,\lambda_2,\ldots,\lambda_j,\ldots\}$ a set of exponents, let $f_1$ and $f_2$ be two almost periodic functions in the class $\mathcal{D}_{\Lambda}$ whose Dirichlet series are given by $\sum_{j\geq 1}a_je^{\lambda_js}$ and $\sum_{j\geq 1}b_je^{\lambda_js}$ respectively. Let $\mathbf{g}:=(g_1,g_2,\ldots,g_k,\ldots)$ be the vector of the elements of the natural basis $G_{\Lambda}^*$ for $\Lambda$. Thus $f_1$ is equivalent to $f_2$ if and only if there exist some $\mathbf{y}\in \mathbb{R}^{\sharp G_\Lambda^*}$ and $\mathbf{p}_j\in 2\pi\mathbb{Z}^{\sharp G_{\Lambda}^*}$, $j=1,2,\ldots$, such that
$$\sum_{j\geq 1}b_je^{\lambda_j(\sigma+it)}=F_{f_1}(\sigma,\mathbf{y}+t\mathbf{g},\mathbf{p}_1,\mathbf{p}_2,\ldots)$$ for $\sigma+it\in U$, where $U$ is an open vertical strip so that $f_2\in AP(U,\mathbb{C})$.
\end{proposition}
\begin{proof}
Let $\sum_{j\geq 1}a_je^{\lambda_js}$ and $\sum_{j\geq 1}b_je^{\lambda_js}$ be the Dirichlet series associated with $f_1$ and $f_2$ respectively. Let $U$ be an open vertical strip so that $f_2\in AP(U,\mathbb{C})$.
If $f_1\shortstack{$_{{\fontsize{6}{7}\selectfont *}}$\\$\sim$} f_2$, then Proposition \ref{propnaturalbasis} assures the existence of
$\mathbf{x}_0\in [0,2\pi)^{\sharp G_\Lambda^*}$ such that for each $j=1,2,\ldots$
it is satisfied $b_j=a_j e^{<\mathbf{r}_j,\mathbf{x}_0+\mathbf{p}_j'>i}$ for some $\mathbf{p}_j'\in 2\pi\mathbb{Z}^{\sharp G_\Lambda^*}$. Thus, fixed $s=\sigma+it\in U$, we have
$$\displaystyle{\sum_{j\geq 1}b_je^{\lambda_j(\sigma+it)}=\sum_{j\geq1}a_je^{i<\mathbf{r}_j,\mathbf{x}_0+\mathbf{p}_j'>}e^{\lambda_j\sigma}e^{i\lambda_j t}}= \sum_{j\geq1}a_je^{\lambda_j\sigma}e^{i<\mathbf{r}_j,\mathbf{x}_0+\mathbf{p}_j'>}e^{it<\mathbf{r}_j,\mathbf{g}>}=$$
$$\sum_{j\geq1}a_je^{\lambda_j\sigma}e^{i<\mathbf{r}_j,\mathbf{x}_0+\mathbf{p}_j'+t\mathbf{g}>}=F_{f_1}(\sigma,\mathbf{y}_0+t\mathbf{g},\mathbf{p}_1,\mathbf{p}_2,\ldots),$$
where $\mathbf{y}_0\in \mathbb{R}^{\sharp G_\Lambda^*}$ and $\mathbf{p}_j\in 2\pi\mathbb{Z}^{\sharp G_{\Lambda}^*}$ are chosen so that $\mathbf{x}_0+t\mathbf{g}+\mathbf{p}_j'=\mathbf{y}_0+t\mathbf{g}+\mathbf{p}_j$, with $\mathbf{y}_0+t\mathbf{g}\in [0,2\pi)^{\sharp G_\Lambda^*}$.

Conversely, suppose the existence of $\mathbf{y}_0\in \mathbb{R}^{\sharp\Lambda}$ and $\mathbf{p}_j\in 2\pi\mathbb{Z}^{\sharp G_{\Lambda}^*}$, $j=1,2,\ldots$,
such that
$\sum_{j\geq 1}b_je^{\lambda_j(\sigma+it)}=F_{f_1}(\sigma,\mathbf{y}_0+t\mathbf{g},\mathbf{p}_1,\mathbf{p}_2,\ldots)$ for any $\sigma+it\in U$. Hence $$\sum_{j\geq 1}b_je^{\lambda_j(\sigma+it)}=\sum_{j\geq1}a_je^{i<\mathbf{r}_j,\mathbf{y}_0+\mathbf{p}_j>}e^{\lambda_j (\sigma+it)}\ \forall \sigma+it\in U.$$
Now, by the uniqueness of the coefficients of an exponential sum in $\mathcal{D}_{\Lambda}$, it is clear that
$b_j=a_j e^{<\mathbf{r}_j,\mathbf{y}_0+\mathbf{p}_j>i}$ for each $j\geq 1$, which shows that $f_1\shortstack{$_{{\fontsize{6}{7}\selectfont *}}$\\$\sim$} f_2$.
\qed
\end{proof}

We next define the following set which will be widely used from now on.

\begin{definition}\label{image}
Given $\Lambda=\{\lambda_1,\lambda_2,\ldots,\lambda_j,\ldots\}$ a set of exponents, let $f(s)\in \mathcal{D}_{\Lambda}$ be an almost periodic function in an open vertical strip $U$, and $\sigma_0=\operatorname{Re}s_0$ with $s_0\in U$. We define $\operatorname{Img}\left(F_f^{G_{\Lambda}}(\sigma_0,\mathbf{x},\mathbf{q}_1,\mathbf{q}_2,\ldots)\right)$ to be the set of values in the complex plane taken on by the auxiliary function $F_f^{G_{\Lambda}}(\sigma,\mathbf{x},\mathbf{q}_1,\mathbf{q}_2,\ldots)$, relative to a prefixed basis $G_{\Lambda}$, when $\sigma=\sigma_0$; that is
$\operatorname{Img}\left(F_f^{G_{\Lambda}}(\sigma_0,\mathbf{x},\mathbf{q}_1,\mathbf{q}_2,\ldots)\right)=\{s\in\mathbb{C}:\exists \mathbf{x}\in[0,2\pi)^{\sharp G_{\Lambda}^*}\ \mbox{and }\mathbf{q}_j\in S_T\mbox{ such that }s=F_f^{G_{\Lambda}}(\sigma_0,\mathbf{x},\mathbf{q}_1,\mathbf{q}_2,\ldots)\}.$
\end{definition}


We next prove that the sets of values taken on by the auxiliary function $F_f^{G_{\Lambda}}(\sigma,\mathbf{x},\mathbf{q}_1,\mathbf{q}_2,\ldots)$ are independent of the basis $G_{\Lambda}$. The proof is similar to that of Corollary \ref{cor1ner}.

\begin{lemma}\label{indep}
Given $\Lambda$ a set of exponents and $G_{\Lambda}$ an arbitrary basis for $\Lambda$, let $f(s)\in \mathcal{D}_{\Lambda}$ be an almost periodic function in an open vertical strip $U$, and $\sigma_0=\operatorname{Re}s_0$ with $s_0\in U$. Then  $$\operatorname{Img}\left(F_f^{G_{\Lambda}}(\sigma_0,\mathbf{x},\mathbf{q}_1,\mathbf{q}_2,\ldots)\right)=\operatorname{Img}\left(F_f^{G_{\Lambda}^*}(\sigma_0,\mathbf{x},\mathbf{p}_1,\mathbf{p}_2,\ldots)\right).$$ 
\end{lemma}
\begin{proof}
Let $\sum_{j\geq 1}a_je^{\lambda_js}$ be the Dirichlet series associated with $f(s)\in \mathcal{D}_{\Lambda}$, and  $G_{\Lambda}^*$ and $G_{\Lambda}$ be the natural and an arbitrary basis for $\Lambda$, respectively. 
For each $j\geq 1$ let $\mathbf{r}_j$ and $\mathbf{s}_j$ be the vector of integer components so that  $\lambda_j=<\mathbf{r}_j,\mathbf{g}>$ and $\lambda_j=<\mathbf{s}_j,\mathbf{h}>$, with $\mathbf{g}$ and $\mathbf{h}$ the vectors associated with the basis $G_{\Lambda}^*$ and $G_{\Lambda}^*$, respectively. Finally, for each integer $k\geq 1$, let $\mathbf{t}_k$ be the vector given by $h_k=<\mathbf{t}_k,\mathbf{g}>$.
Take $w_1\in \operatorname{Img}\left(F_f^{G_{\Lambda}^*}(\sigma_0,\mathbf{x},\mathbf{p}_1,\mathbf{p}_2,\ldots)\right)$, then there exists $\mathbf{x}_1\in[0,2\pi)^{\sharp G_{\Lambda}}$ and $\mathbf{p}_j\in 2\pi\mathbb{Z}^{\sharp G_{\Lambda}^*}$, $j=1,2,\ldots$, $\mbox{ such that }w_1=F_f^{G_{\Lambda}^*}(\sigma_0,\mathbf{x}_1,\mathbf{p}_1,\mathbf{p}_2,\ldots)$. Hence
$$w_1=F_f^{G_{\Lambda}^*}(\sigma_0,\mathbf{x}_1,\mathbf{p}_1,\mathbf{p}_2,\ldots)=\sum_{j\geq1}a_j e^{\lambda_j\sigma_0
}e^{<\mathbf{r}_j,\mathbf{x}_1+\mathbf{p}_j>i}=\sum_{j\geq1}a_j e^{\lambda_j\sigma_0
}e^{<\mathbf{s}_j,\mathbf{x}_2+\mathbf{q}_j>i},$$
where $\mathbf{q}_j$ is defined as $q_{1,k}=<\mathbf{t}_k,\mathbf{p}_j>$ for each $k\geq 1$, and $\mathbf{x}_2$ is defined as $x_{2,k}=<\mathbf{t}_k,\mathbf{x}_1>$ for each $k\geq 1$. Therefore, $w_1=F_f^{G_{\Lambda}}(\sigma_0,\mathbf{x}_2,\mathbf{q}_1,\mathbf{q}_2,\ldots)$ and $w_1\in \operatorname{Img}\left(F_f^{G_{\Lambda}}(\sigma_0,\mathbf{x},\mathbf{q}_1,\mathbf{q}_2,\ldots)\right)$, which gives $$\operatorname{Img}\left(F_f^{G_{\Lambda}^*}(\sigma_0,\mathbf{x},\mathbf{p}_1,\mathbf{p}_2,\ldots)\right)\subseteq \operatorname{Img}\left(F_f^{G_{\Lambda}}(\sigma_0,\mathbf{x},\mathbf{q}_1,\mathbf{q}_2,\ldots)\right).$$ An analogous argument shows that
$\operatorname{Img}\left(F_f^{G_{\Lambda}}(\sigma_0,\mathbf{x},\mathbf{q}_1,\mathbf{q}_2,\ldots)\right)$ is included in $ \operatorname{Img}\left(F_f^{G_{\Lambda}^*}(\sigma_0,\mathbf{x},\mathbf{p}_1,\mathbf{p}_2,\ldots)\right)$, which proves the result.
\qed
\end{proof}

Consequently, from now on we will use the notation $\operatorname{Img}\left(F_f(\sigma_0,\mathbf{x},\mathbf{p}_1,\mathbf{p}_2,\ldots)\right)$ for the set of values taken on by the auxiliary function associated with a function $f(s)\in \mathcal{D}_{\Lambda}$. In this respect, without loss of generality, we can use the natural basis for the set of exponents $\Lambda$. Moreover, under the assumption of existence of an integral basis, we will use the notation $\operatorname{Img}\left(F_f(\sigma_0,\mathbf{x})\right)$ for the set above. In fact, in this case $\operatorname{Img}\left(F_f(\sigma_0,\mathbf{x})\right)$ is the same as that of \cite[Definition 5]{SVBohr} and all the results of \cite{SVBohr} concerning integral basis are also valid for our case (see also Remark 7 on the Arxiv version of \cite{SV}). 

\section{Main results}

Given a function $f(s)$, take the notation $$\operatorname{Img}\left(f(\sigma_0+it)\right)=\{s\in\mathbb{C}:\exists t\in\mathbb{R}\mbox{ such that }s=f(\sigma_0+it)\}.$$

We next show the first important result in this paper concerning the connection between our equivalence relation and the set of values in the complex
plane taken on by the auxiliary function (compare with \cite[Proposition 2]{SVBohr} for the case of existence of an integral basis).

\begin{proposition}\label{pult}
Given $\Lambda$ a set of exponents, let $f(s)\in \mathcal{D}_{\Lambda}$ be an almost periodic function in an open vertical strip $U$, and $\sigma_0=\operatorname{Re}s_0$ with $s_0\in U$.
\begin{itemize}
\item[i)] If $f_1\shortstack{$_{{\fontsize{6}{7}\selectfont *}}$\\$\sim$} f$, then $\operatorname{Img}\left(f_1(\sigma_0+it)\right)\subset \overline{\operatorname{Img}\left(f(\sigma_0+it)\right)}$ and $$\overline{\operatorname{Img}\left(f(\sigma_0+it)\right)}= \overline{\operatorname{Img}\left(f_1(\sigma_0+it)\right)}.$$

\item[ii)] $\operatorname{Img}\left(F_f(\sigma_0,\mathbf{x},\mathbf{p}_1,\mathbf{p}_2,\ldots)\right)=\bigcup_{f_k\shortstack{$_{{\fontsize{6}{7}\selectfont *}}$\\$\sim$} f}\operatorname{Img}\left(f_k(\sigma_0+it)\right).$

\item[iii)] $\operatorname{Img}\left(F_f(\sigma_0,\mathbf{x},\mathbf{p}_1,\mathbf{p}_2,\ldots)\right)$ is a closed set.\\

\item[iv)] $\operatorname{Img}\left(F_f(\sigma_0,\mathbf{x},\mathbf{p}_1,\mathbf{p}_2,\ldots)\right)=\overline{\operatorname{Img}\left(f_1(\sigma_0+it)\right)}$ for any $f_1\sim f$.
\end{itemize}
\end{proposition}
\begin{proof}

i) Note that \cite[Theorem 4]{SV} (see also the Arxiv version) 
shows that the functions in the same equivalence class are obtained as limit points of $\mathcal{T}_f=\{f_{\tau}(s):=f(s+i\tau):\tau\in\mathbb{R}\}$, that is, any function $f_1\shortstack{$_{{\fontsize{6}{7}\selectfont *}}$\\$\sim$} f$ is the limit (in the sense of the uniform convergence on every reduced strip of $U$) of a sequence $\{f_{\tau_n}(s)\}$ with $f_{\tau_n}(s):=f(s+i\tau_n)$.
    Take $w_1\in \operatorname{Img}\left(f_1(\sigma_0+it)\right)$, then there exists $t_1\in\mathbb{R}$ such that $w_1=f_1(\sigma_0+it_1)$. Now, given $\varepsilon>0$ there exists $\tau>0$ such that $|f_1(\sigma_0+it_1)-f_{\tau}(\sigma_0+it_1)|<\varepsilon$, which means that
$$ |w_1-f(\sigma_0+i(t_1+\tau))|<\varepsilon.$$
Now it is immediate that $w_1\in \overline{\operatorname{Img}\left(f(\sigma_0+it)\right)}$ and consequently $$\operatorname{Img}\left(f_1(\sigma_0+it)\right)\subset \overline{\operatorname{Img}\left(f(\sigma_0+it)\right)}.$$
Analogously, by symmetry we have $\operatorname{Img}\left(f(\sigma_0+it)\right)\subset \overline{\operatorname{Img}\left(f_1(\sigma_0+it)\right)}$, which implies that     $$\overline{\operatorname{Img}\left(f(\sigma_0+it)\right)}= \overline{\operatorname{Img}\left(f_1(\sigma_0+it)\right)}.$$

ii)
Take $w_0\in \bigcup_{f_k\shortstack{$_{{\fontsize{6}{7}\selectfont *}}$\\$\sim$} f}\operatorname{Img}\left(f_k(\sigma_0+it)\right)$, then $w_0\in \operatorname{Img}\left(f_k(\sigma_0+it)\right)$ for some $f_k\shortstack{$_{{\fontsize{6}{7}\selectfont *}}$\\$\sim$} f$, which means that there exists $t_0\in\mathbb{R}$ such that $$w_0=f_k(\sigma_0+it_0).$$ Note that (see also \cite[Remark 1]{SVBohr}) Proposition \ref{lequiv20} assures the existence of a vector $\mathbf{y}_0\in \mathbb{R}^{\sharp G_\Lambda^*}$ and $\mathbf{p}_j\in 2\pi\mathbb{Z}^{\sharp G_{\Lambda}^*}$, $j=1,2,\ldots$, such that
$w_0=F_{f}(\sigma,\mathbf{y}_0+t_0\mathbf{g},\mathbf{p}_1,\mathbf{p}_2,\ldots)$
Hence $w_0=F_{f}(\sigma_0,\mathbf{x}_0,\mathbf{p}_1,\mathbf{p}_2,\ldots)$, with $\mathbf{x}_0=\mathbf{y}_0+t_0\mathbf{g}\in[0,2\pi)^{\sharp G_{\Lambda}^*}$, which means that $w_0\in \operatorname{Img}\left(F_f(\sigma_0,\mathbf{x},\mathbf{p}_1,\mathbf{p}_2,\ldots)\right)$. Conversely, if $w_0\in \operatorname{Img}\left(F_f(\sigma_0,\mathbf{x},\mathbf{p}_1,\mathbf{p}_2,\ldots)\right)$, then
$w_0=F_{f}(\sigma_0,\mathbf{y}_0,\mathbf{p}_1,\mathbf{p}_2,\ldots)$ for some $\mathbf{y}_0\in[0,2\pi)^{\sharp G_{\Lambda}}$ and $\mathbf{p}_j\in 2\pi\mathbb{Z}^{\sharp G_{\Lambda}^*}$. Take $t_0\in\mathbb{R}$. Since $\mathbf{y}_0=\mathbf{x}_0+t_0\mathbf{g}$, with $\mathbf{x}_0:=\mathbf{y}_0-t_0\mathbf{g}$, then
$$w_0=F_{f}(\sigma_0,\mathbf{x}_0+t_0\mathbf{g},\mathbf{p}_1,\mathbf{p}_2,\ldots)=\sum_{j\geq1}a_j e^{\lambda_j\sigma_0
}e^{<\mathbf{r}_j,\mathbf{x}_0+t_0\mathbf{g}+\mathbf{p}_j>i}=$$$$\sum_{j\geq1}a_j e^{\lambda_j(\sigma_0+it_0)
}e^{<\mathbf{r}_j,\mathbf{x}_0+\mathbf{p}_j>i}.$$
Hence $\sum_{j\geq1}a_je^{<\mathbf{r}_j,\mathbf{x}_0+\mathbf{p}_j>i} e^{\lambda_js}$ is the associated Dirichlet series of an almost periodic function $h(s)\in AP(U,\mathbb{C})$ such that $h\shortstack{$_{{\fontsize{6}{7}\selectfont *}}$\\$\sim$} f$ (see the Arxiv version of \cite[Lemma 3]{SV}) and hence we have that $w_0=h(\sigma_0+it_0)$ (see also \cite[Remark 1]{SVBohr}), which shows that $w_0\in \bigcup_{f_k\shortstack{$_{{\fontsize{6}{7}\selectfont *}}$\\$\sim$} f}\operatorname{Img}\left(f_k(\sigma_0+it)\right).$


iii) Let $w_1,w_2,\ldots,w_j,\ldots$ be a sequence of points in $\operatorname{Img}\left(F_f(\sigma_0,\mathbf{x},\mathbf{p}_1,\mathbf{p}_2,\ldots)\right)$
tending to $w_0$. We next prove that  $w_0\in \operatorname{Img}\left(F_f(\sigma_0,\mathbf{x},\mathbf{p}_1,\mathbf{p}_2,\ldots)\right)$. Indeed, for each
$w_j\in \operatorname{Img}\left(F_f(\sigma_0,\mathbf{x},\mathbf{p}_1,\mathbf{p}_2,\ldots)\right)$, by ii), there exists $f_j\shortstack{$_{{\fontsize{6}{7}\selectfont *}}$\\$\sim$} f$ such that $w_j\in \operatorname{Img}\left(f_j(\sigma_0+it)\right)$. Now, since that $\{f_j(\sigma_0+it)\}$ is a sequence in the same equivalence class, \cite[Proposition 3]{SV} (see also the Arxiv version) assures the existence of a subsequence $\{f_{j_k}\}$ which converges to a certain function $h\shortstack{$_{{\fontsize{6}{7}\selectfont *}}$\\$\sim$} f$. Consequently, $\{w_{j_k}\}$ tends to $w_0\in \operatorname{Img}\left(h(\sigma_0+it)\right)$. Finally, again by ii) we conclude that $w_0\in \operatorname{Img}\left(F_f(\sigma_0,\mathbf{x},\mathbf{p}_1,\mathbf{p}_2,\ldots)\right)$.

iv) Let $\mathbf{g}$ be the vector associated with the natural basis $G_{\Lambda}^*$. Since the Fourier series of $f_{\sigma_0}(t):=f(\sigma_0+it)$ can be obtained as $F_f(\sigma_0,t\mathbf{g},\mathbf{0},\mathbf{0},\ldots)$, with $t\in\mathbb{R}$, it is clear that $\operatorname{Img}\left(f(\sigma_0+it)\right)\subset \operatorname{Img}\left(F_f(\sigma_0,\mathbf{x},\mathbf{p}_1,\mathbf{p}_2,\ldots)\right)$. On the other hand,  we deduce from i) and ii)  that
    $$\operatorname{Img}\left(f(\sigma_0+it)\right)\subset \operatorname{Img}\left(F_f(\sigma_0,\mathbf{x},\mathbf{p}_1,\mathbf{p}_2,\ldots)\right)=
    \bigcup_{f_k\shortstack{$_{{\fontsize{6}{7}\selectfont *}}$\\$\sim$} f}\operatorname{Img}\left(f_k(\sigma_0+it)\right)\subset$$$$ \subset\overline{\operatorname{Img}\left(f(\sigma_0+it)\right)}.$$
    Finally, by taking the closure and property iii),
    we conclude that
    $$\operatorname{Img}\left(F_f(\sigma_0,\mathbf{x},\mathbf{p}_1,\mathbf{p}_2,\ldots)\right)=\overline{\operatorname{Img}\left(f(\sigma_0+it)\right)}.$$
Now, the result follows from property i).
\qed
\end{proof}

At this point we will demonstrate a result like Bohr's equivalence theorem \cite[Section 8.11]{Apostol}. Given $\Lambda$ an arbitrary set of exponents, let $f_1,f_2\in \mathcal{D}_{\Lambda}$ be two equivalent almost periodic functions. We next show that, in any open half-plane or open vertical strip included in their region of almost periodicity, the functions $f_1$ and $f_2$ take the same set of values. In this sense, this result improves that of \cite[Theorem 1]{SVBohr} which was proved uniquely for almost periodic functions associated with sets of exponents which have an integral basis.

\begin{theorem}\label{beqg}
Given $\Lambda$ a set of exponents, let $f_1,f_2\in \mathcal{D}_{\Lambda}$ be two equivalent almost periodic functions in a vertical strip $\{\sigma+it\in\mathbb{C}:\alpha<\sigma<\beta\}$. Consider $E$ an open set of real numbers included in $(\alpha,\beta)$.
Thus $$\bigcup_{\sigma\in E}\operatorname{Img}\left(f_1(\sigma+it)\right)=\bigcup_{\sigma\in E}\operatorname{Img}\left(f_2(\sigma+it)\right).$$
That is, the functions $f_1$ and $f_2$ take the same set of values on the region $\{s=\sigma+it\in\mathbb{C}:\sigma\in E\}$.
\end{theorem}
\begin{proof} 
Without loss of generality, suppose that $f_1$ and $f_2$ are not constant functions (otherwise it is trivial). Take $w_0\in \bigcup_{\sigma\in E}\operatorname{Img}\left(f_1(\sigma+it)\right)$, then $w_0\in \operatorname{Img}\left(f_1(\sigma_0+it)\right)$ for some $\sigma_0\in E$ and hence $w_0=f_1(\sigma_0+it_0)$ for some $t_0\in\mathbb{R}$. Furthermore, by Proposition \ref{pult}, we get
$w_0\in\overline{\operatorname{Img}\left(f_1(\sigma_0+it)\right)}=\overline{\operatorname{Img}\left(f_2(\sigma_0+it)\right)}$,
which implies that there exists a sequence $\{t_n\}$ of real numbers such that
$$w_0=\lim_{n\to\infty}f_2(\sigma_0+it_n).$$ Take $h_n(s):=f_2(s+it_n)$, $n\in\mathbb{N}$.
By \cite[Proposition 4]{SV} (see also the Arxiv version), there exists a subsequence $\{h_{n_k}\}_k\subset \{h_n\}_n$ which converges uniformly on compact subsets to a function $h(s)$, with $h\shortstack{$_{{\fontsize{6}{7}\selectfont *}}$\\$\sim$} f_2$. Observe that $$\lim_{k\to\infty}h_{n_k}(\sigma_0)=h(\sigma_0)=w_0.$$ Therefore, by Hurwitz's theorem \cite[Section 5.1.3]{Ash2}, there is a positive integer $k_0$ 
such that for $k>k_0$ the functions $h^*_{n_k}(s):=h_{n_k}(s)-w_0$ have one zero in $D(\sigma_0,\varepsilon)$ for any $\varepsilon>0$ sufficiently small. This means that for $k>k_0$ the functions $h_{n_k}(s)=f_2(s+it_{n_k})$, and hence the function $f_2(s)$, take the value $w_0$ on the region $\{s=\sigma+it:\sigma_0-\varepsilon<\sigma<\sigma_0+\varepsilon\}$ for any $\varepsilon>0$ sufficiently small (recall that $E$ is an open set). 
Consequently, $w_0\in \bigcup_{\sigma\in E}\operatorname{Img}\left(f_2(\sigma+it)\right)$.
We analogously prove that $\bigcup_{\sigma\in E}\operatorname{Img}\left(f_2(\sigma+it)\right)\subset \bigcup_{\sigma\in E}\operatorname{Img}\left(f_1(\sigma+it)\right)$.
\qed
\end{proof}

Finally, we note that \cite[Example 2]{SVBohr} also shows that, fixed an open set $E$ in $(\alpha,\beta)$, the converse of Theorem \ref{beqg} is not true.

\bibliographystyle{amsplain}

\end{document}